\documentclass[reqno,a4paper
]{amsart} 

\usepackage{hyperref}

\hypersetup{pdfauthor={Matthias Lenz},pdftitle={Lattice points in polytopes, box splines, and Todd operators}, 
  pdfsubject={Zonotopal Algebra}}
 
 \usepackage[T1,T2A]{fontenc}
 \usepackage[ansinew]{inputenc}
\usepackage{amssymb} 
\usepackage{upref} 
\usepackage{graphicx}
\usepackage{color}
\usepackage{url}
\usepackage{paralist} 

\usepackage{wasysym} 

\usepackage[russian,british]{babel}
\usepackage{mathtext} 
\DeclareSymbolFont{T2Aletters}{T2A}{cmr}{m}{it} 

\usepackage{calc}
\usepackage{bm} 

\usepackage[all]{xy}

\newtheorem{Theorem}{Theorem}
\newtheorem{Corollary}[Theorem]{Corollary}
\newtheorem{Proposition}[Theorem]{Proposition}
\newtheorem{Lemma}[Theorem]{Lemma}

\theoremstyle{definition}
\newtheorem{Definition}[Theorem]{Definition}

\newtheorem{Example}[Theorem]{Example}

\theoremstyle{remark}
\newtheorem{Remark}[Theorem]{Remark}

\newcommand{\di}{\mathrm{d}} 
\newcommand{\DM}{\mathop{\mathrm{DM}}\nolimits}
\newcommand{\supp}{\mathop{\mathrm{supp}}\nolimits}

\newcommand{\vol}{\mathop{\mathrm{vol}}}
\newcommand{\abs}[1]{\left|#1\right|}
\newcommand{\rank}{\mathop{\mathrm{rk}}}

\newcommand{\st}{s.\,t.\ } 
\newcommand{\ie}{\textit{i.\,e.\ }} 
\newcommand{\eg}{\textit{e.\,g.\ }} 

\newcommand{\Z}{\mathbb{Z}}

\newcommand{\R}{\mathbb{R}}
\newcommand{\CC}{\mathbb{C}}

\newcommand{\Bcal}{\mathcal{B}}

\newcommand{\Dcal}{\mathcal{D}}

\newcommand{\Zcal}{\mathcal{Z}}
\newcommand{\Pcal}{\mathcal{P}}
\newcommand{\Tcal}{\mathcal{T}}

\newcommand{\Jcal}{\mathcal{J}}

\newcommand{\eps}{\varepsilon}

\newcommand{\Bcyr}{\text{\textit{\CYRB}}} 

\newcommand{\res}{\mathop{\mathrm{Res}}}
\newcommand{\todd}{\mathop{\mathrm{Todd}}}
\newcommand{\ideal}{\mathop{\mathrm{ideal}}}
\newcommand{\spa}{\mathop{\mathrm{span}}}
\newcommand{\cone}{\mathop{\mathrm{cone}}}
\newcommand{\BB}{ \mathbb B}

\newcommand{\cfrak}{\mathfrak c}

\newcommand{\diff}[1]{\frac{\partial}{\partial #1}}

\newcommand{\sym}{\mathop{\mathrm{Sym}}\nolimits} 
\newcommand{\pair}[2]{\langle #1,#2 \rangle}
\newcommand{\interior}{\mathop{\mathrm{int}}}

\title[{Lattice points, box splines, and Todd operators}]{Lattice points in polytopes, box splines, and Todd operators}
\author{Matthias Lenz}
\email{lenz@maths.ox.ac.uk}
\address{%
Mathematical Institute\\
24-29 St Giles'\\
Oxford\\
OX1 3LB\\
United Kingdom
}
\thanks{The author was supported by a Junior Research Fellowship
 of Merton College (University of Oxford).
 }
\date{\today}


\newcommand{\Xintro}{%
Let $X\subseteq U\cong \R^d$ be a finite list of vectors that spans $U$\!.
}

\begin{document}

\begin{abstract}
Let $X$ be a list of vectors that is totally unimodular.
In a previous article the author proved that every real-valued function on the
 set of interior lattice points 
 of the zonotope defined by $X$ can be extended to a function on
the whole zonotope of the form
$p(D)B_X$ in a unique way, where
$p(D)$ is a differential operator that is contained in the
so-called internal $\Pcal$-space. 
In this paper we construct an explicit solution to this interpolation problem in terms of Todd operators.
 As a corollary we obtain a slight generalisation of the Khovanskii-Pukhlikov formula that relates the volume
  and the number of integer points in a smooth lattice polytope. 
\end{abstract}
\maketitle

%
%
%
%

\section{Introduction}
Box splines and multivariate splines measure the volume of certain variable polytopes. The vector partition function that 
 measures the number of integral points in polytopes can be seen as a discrete version of these spline functions.
 Splines and vector partition functions have recently received a lot of attention by researchers in various fields including approximation theory,
algebra, combinatorics, and representation theory. A standard reference from the 
approximation theory point of view is the book \cite{BoxSplineBook} by
 de Boor, H\"ollig, and Riemenschneider. The combinatorial and algebraic aspects are stressed in the book  \cite{concini-procesi-book}
 by De Concini and Procesi.

 Khovaniskii and Pukhlikov proved a remarkable formula that relates the volume and the number of integer points in a smooth polytope \cite{pukhlikov-khovanski-1992}.
  The connection is made via Todd operators, 
  \ie differential operators of type $\frac{\partial_x}{1 - e^{\partial_x}}$.
 The formula is closely related to the Hirzebruch-Riemann-Roch Theorem for smooth projective toric varieties %
(see \cite[Chapter 13]{cox-little-schenck-2011}).   
  De Concini, Procesi, and Vergne have shown that the  Todd operator is in a certain sense inverse
 to convolution with the box spline \cite{deconcini-procesi-vergne-2010c}. This implies the Khovanskii-Pukhlikov formula and 
 more generally the formula of Brion-Vergne \cite{brion-vergne-1997}.

 In this paper we will prove a slight generalisation of the deconvolution formula by De Concini, Procesi, and Vergne. 
 The operator that we  use is obtained from the Todd operator, but it is simpler, \ie it a polynomial contained in the 
 so-called internal $\Pcal$-space.
 Our proof uses deletion-contraction, so in some sense we provide a matroid-theoretic proof of the Khovanskii-Pukhlikov formula.

 Furthermore, we will construct bases for the spaces $\Pcal_-(X)$
 and $\Pcal(X)$ that 
 were studied by Ardila and Postnikov in connection with power ideals \cite{ardila-postnikov-2009} and by Holtz and Ron 
 within the framework  of zontopal algebra \cite{holtz-ron-2011}. 
  Up to now,  no 
 general construction for a basis of the space $\Pcal_-(X)$ was known.

\medskip
 Let us introduce our notation. It is similar to the one used in \cite{concini-procesi-book}.
  We fix a $d$-dimensional real vector space $U$ and a lattice $\Lambda\subseteq U$.
  Let $X=(x_1,\ldots, x_N) \subseteq \Lambda$ be a finite list of vectors that spans $U$. 
 We assume that $X$ is totally unimodular with respect to $\Lambda$, \ie every basis for $U$ that can be selected from $X$ is also
 a lattice basis for $\Lambda$.
Note that $X$ can be identified with a linear map $X : \R^N \to U$. 
Let $u\in U$. We define the variable polytopes
\begin{align}
 \Pi_X(u) &:=  \{ w \in \R^N_{\ge 0} :   X w = u  \} 
 \quad \text{ and } \quad  \Pi^1_X(u) :=  \Pi_X(u) \cap [0;1]^N.
\end{align}
Note that any convex polytope can be written in the form  $\Pi_X(u)$ for suitable $X$ and $u$.
 The dimension of these two polytopes is at most $N-d$. We define
the 
\begin{align}
\text{\emph{vector partition function} }
\Tcal_X(u) &:= \abs{\Pi_X(u) \cap \Z^N}, \\
\text{ the \emph{box spline} } B_X(u) &:= \det(XX^T)^{-1/2}\vol\nolimits_{N-d}{\Pi^1_X(u)}, \\
\text{
 and the \emph{multivariate spline} }
 T_X(u) &:= \det(XX^T)^{-1/2}\vol\nolimits_{N-d}{\Pi_X(u)}.
 \end{align}
 Note that we have to assume that $0$ is not contained in the convex hull of $X$ in order for $T_X$ and $\Tcal_X$ to be well-defined.
 Otherwise,  $\Pi_X(u)$ is an unbounded polyhedron.
The \emph{zonotope} $Z(X)$ is defined as
 \begin{equation}
 Z(X):= \left\{ \sum_{i=1}^N \lambda_i x_i : 0\le \lambda_i \le 1  \right\} = X \cdot [0,1]^N.
 \end{equation}
We denote its set of interior lattice points by $\Zcal_-(X) := \interior(Z(X)) \cap \Lambda$. 
 The symmetric algebra over $U$ is denoted by $\sym(U)$.
  We fix a basis $s_1,\ldots, s_d$ for the lattice $\Lambda$. This makes it possible to identify $\Lambda$ with $\Z^d$,
  $U$ with $\R^d$, $\sym(U)$ with the polynomial ring
   $\R[s_1,\ldots, s_d]$, and $X$ with a $(d\times N)$-matrix. Then
   $X$ is totally unimodular if and only if every non-singular square submatrix of this matrix has determinant $1$ or $-1$.
 The base-free setup is more convenient when working with quotient vector spaces.
 
 We denote the dual vector space by $V = U^*$ and we fix a basis $t_1,\ldots, t_d$ that is dual to the basis for 
  $U$. An element of $\sym(U)$ can be seen as a differential operator on $\sym(V)$,  \ie
   $\sym(U) \cong \R[s_1,\ldots, s_d] \cong \R[\diff{t_1},\ldots, \diff{t_d}]$.
   For $f\in \sym(U)$ and $p \in \sym(V)$ we  write $f(D) p$ to denote the polynomial 
   in $\sym(V)$ that is obtained when $f$ acts on $p$ as a differential operator.
  It is known that the box spline is piecewise polynomial and its local pieces are contained in $\sym(V)$. 
  We will mostly use elements of $\sym(U)$ as differential operators 
   on its local pieces.

Note that a vector $u\in U$ defines a linear form $u \in \sym(U)$.
For a sublist $Y\subseteq X$, we define $p_Y := \prod_{y\in Y} y$.  For example, if $Y=((1,0),(1,2))$, then 
 $p_Y=s_1^2 + 2s_1s_2$. Furthermore, $p_\emptyset := 1$.
 Now we define the  
\begin{align}
\text{\emph{central $\Pcal$-space} } \Pcal(X) &:= \spa\{ p_Y :  \rank(X\setminus Y)= \rank(X) \} \\
\text{and the \emph{internal $\Pcal$-space} }   \Pcal_-(X) &:= \bigcap_{x\in X} \Pcal(X\setminus x).
\end{align}
The space $\Pcal_-(X)$ was introduced in \cite{holtz-ron-2011} where it was also shown that the dimension of this space is equal to 
 $\abs{\Zcal_-(X)}$.
The space $\Pcal(X)$ first appeared in approximation theory \cite{akopyan-saakyan-1988,boor-dyn-ron-1991,dyn-ron-1990}. 
 These two $\Pcal$-spaces and generalisations were later studied by various authors, including
  \cite{ardila-postnikov-2009,berget-2010,holtz-ron-xu-2012,lenz-hzpi-2012,lenz-forward-2012, li-ron-2011}.%

In \cite{lenz-interpolation-2012}, the author proved the following theorem, which will be made more explicit in the present paper.
\begin{Theorem}
\label{Theorem:weakHoltzRon}
Let $X\subseteq \Lambda\subseteq U\cong \R^d$ 
be a list of vectors that is totally unimodular and spans $U$. %
Let $f$ be a real valued function on $\Zcal_-(X)$, 
the set of interior lattice points of the zonotope defined by $X$.

Then there exists a unique polynomial $p\in \Pcal_-(X)\subseteq\R[s_1,\ldots, s_d]$, \st 
 $p(D)B_X$ is a continuous function and its restriction to $\Zcal_-(X)$ is  equal to $f$. %

\smallskip
Here, $p(D)$ denotes the differential operator obtained from $p$ by replacing  
 the variable $s_i$ by $\diff{s_i}$
and $B_X$  denotes  the box spline defined by $X$.
\end{Theorem}

Let $z \in U$. As usual, the exponential is defined as $e^z:=\sum_{k\ge 0} \frac{z^k}{k!} \in \R[[s_1,\ldots, s_d]]$.
We define the ($z$-shifted) \emph{Todd operator}
 \begin{align}
  \todd(X,z) :=  e^{-z} \prod_{x\in X} \frac{x}{1 - e^{-x}} \in \R[[s_1,\ldots, s_d]].
 \end{align}
The Todd operator can be expressed in terms of   
the \emph{Bernoulli numbers}  $B_0= 1$, $B_1 = - \frac 12 $, $B_2 = \frac 16$,$\ldots $ Recall that they 
are defined by the equation:
\begin{align}
 \frac{s}{e^s-1} = \sum_{k\ge 0} \frac{B_k}{k!} s^k. %
\end{align}
One should note that 
 $e^{z}\frac{z}{e^z-1} = \frac{z}{1-e^{-z}} = \sum_{k\ge 0} \frac{B_k}{k!} (-z)^k$.
For $z\in \Zcal_-(X)$ we can fix a list $S\subseteq X$ \st $z=\sum_{x\in S} x$. Let $T:=X\setminus S$. 
Then we can write the Todd operator as
\begin{align*}
 \todd(X,z) = \prod_{x\in S} \left(\sum_{k\ge 0} \frac{B_k}{k!} x^k\right)
    \prod_{x\in T}\left(\sum_{k\ge 0} \frac{B_k}{k!} (-x)^k\right)
  = 1 + 
  \sum_{x\in T} \frac x2 - \sum_{x\in S} \frac x2   + \ldots
\end{align*}

A sublist $C\subseteq X$ is called a \emph{cocircuit} if  $\rank(X\setminus C) < \rank(X) $ and $C$ is inclusion minimal with this property.
We consider the \emph{cocircuit ideal} 
$\Jcal(X) := \ideal \{ p_C  : C \text{ cocircuit}\} \subseteq \sym(U)$. It is known \cite{dyn-ron-1990, holtz-ron-2011} that 
$\sym(U) = \Pcal(X) \oplus \Jcal(X)$.
Let 
\begin{align}
\label{eq:ProjectionMap}
 \psi_X : \Pcal(X) \oplus \Jcal(X) \to \Pcal(X) 
\end{align}
 denote the projection. 
 Note that this is a graded linear map and that $\psi_X$ maps any homogeneous polynomial  to zero whose degree is at least $N-d+1$.
  This implies that there is a canonical extension 
$ \psi_X : \R[[s_1,\ldots, s_d ]] \to \Pcal(X)$  given by $\psi_X(\sum_i( g_i)) :=
 \sum_i \psi_X(  g_i )$, where $g_i$ denotes a homogeneous polynomial of degree~$i$.
Let
\begin{align}
 f_z = f_z^X :=\psi_X(\todd(X,z)).
\end{align}

\begin{Theorem}[Main Theorem]
\label{Theorem:MainTheorem}
 Let $X\subseteq \Lambda\subseteq U \cong \R^d$ be a list of vectors that is totally unimodular and spans $U$. Let $z$
 be a lattice point in the interior of 
 the zonotope $Z(X)$.
 Then $f_z\in \Pcal_-(X)$, $\todd(X,z)(D)B_X$ extends continuously on $U$, and
 \begin{align}
 \label{eq:MainThm}
    f_z (D)  B_X |_\Lambda = 
    \todd(X,z) (D) B_X |_\Lambda = \delta_z  %
.
 \end{align}
 Here, $f_z$ and $\todd(X,z)$ act on the box spline $B_X$ as  differential operators.
\end{Theorem}
 Dahmen and Micchelli observed that  
 \begin{align}
 \label{eq:DahmenMicchelli}
   T_X = B_X *_d  \Tcal_X := \sum_{\lambda\in \Lambda } B_X(\cdot - \lambda)\Tcal_X(\lambda)  
 \end{align}
 (cf.~\cite[Proposition 17.17]{concini-procesi-book}). 
Using this result, the following variant of the
 Khovanskii-Pukhlikov formula \cite{pukhlikov-khovanski-1992} follows immediately.
 \begin{Corollary}
   \label{Corollary:KhovanskiiPukhlikov}
     Let $X\subseteq \Lambda \subseteq U \cong \R^d$ be a 
      list of vectors that is totally unimodular and spans $U$, $u\in \Lambda$ and 
      $ z \in \Zcal_-(X) $. Then
   \begin{align}
    \abs{\Pi_X(u - z)\cap \Lambda } = 
     \Tcal_X(u - z) = \todd(X,z)(D) T_X(u) = f_z(D)  T_X(u). 
   \end{align}
  \end{Corollary}

\begin{Remark}
 The box spline $B_X$ is piecewise polynomial. Hence 
 each of its local pieces is smooth but the whole function
 is not smooth where two different regions of 
  polynomiality intersect.
 De Concini, Procesi, and Vergne  \cite{deconcini-procesi-vergne-2010c}
 proved the following \emph{deconvolution formula}, where $B_X$ is replaced by a suitable local piece $p_\cfrak$:
 $\todd(X,0)(D)p_\cfrak  |_\Lambda  = \delta_0$.
 In Section~\ref{Section:ProofIdea} we will explain the choice of the local piece.

 One can deduce from \cite[Remark 3.15]{deconcini-procesi-vergne-2010c}
 that $\todd(X,z)(D) B_X$ can be extended continuously if $z\in \Zcal_-(X)$. 
 It is also not difficult to show that 
 $\todd(X,z) (D)  B_X = f_z (D) B_X$  (see Lemma~\ref{Lemma:BoxSplineDiffPspace}) 
 and that multiplying the Todd operator by $e^{-x}$ corresponds to translating $\todd(X,z) (D)  B_X$ by $x$.
 The novelty of the Main Theorem is that
 the operator $f_z$ for $z\in \Zcal_-(X)$ is shorter than the original Todd operator
 operator (cf.\ Example~\ref{Example:SmoothPolygon}), \ie it is contained in $\Pcal_-(X)$. 
 Furthermore, we provide a new proof for De Concini, Procesi, and Vergne's deconvolution formula.
\end{Remark}
We will also prove a slightly different version  of the Main Theorem 
(Theorem~\ref{Theorem:MainTheoremBoundary}) 
 that only holds for local pieces of the box spline but
 where lattice points in the boundary of the zonotope are permitted  as well.
 This theorem implies the following result. 
   \begin{Corollary}
   \label{Corollary:KV}
     Let $X\subseteq \Lambda \subseteq U \cong \R^d$ be a 
      list of vectors that is totally unimodular and spans $U$ and let $z$ be a lattice point in the zonotope $Z(X)$.
       Let $u\in \Lambda$ and let $\Omega\subseteq \cone(X)$  be a chamber \st $u$ is contained in its closure.
        Let $p_\Omega$ be the polynomial that agrees with $T_X$ on $\Omega$.
      Then
       \begin{align}
    \abs{ \Pi_X(u - z) \cap \Lambda } = \Tcal_X ( u - z ) = \todd(X,z)(D) p_\Omega(u) = f_z(D) p_\Omega(u).
   \end{align}
  \end{Corollary}
The original Khovanskii-Pukhlikov formula is the case $z=0$ in Corollary~\ref{Corollary:KV}. %
For more information on this formula, see Vergne's survey article on integral points in polytopes \cite{vergne-2003}. 
 An explanation of the Khovanskii-Pukhlikov formula that is easy to read is contained 
  in the book by Beck and Robins 
  \cite[Chapter 10]{beck-robins-ComputingTheContinuousDiscretely}.

  \begin{Corollary}
  \label{Corollary:DahmenMicchelli}
     Let $X\subseteq \Lambda \subseteq U \cong \R^d$ be a list of vectors that is totally unimodular and spans $U$.
  Then 
  \begin{align}
    \sum_{z \in \Zcal_-(X)} B_X(z) f_z = 1. 
  \end{align}
   This implies formula \eqref{eq:DahmenMicchelli}.
 \end{Corollary}
The central $\Pcal$-space and various other generalised $\Pcal$-spaces have a canonical basis \cite{holtz-ron-2011,lenz-forward-2012}.
Up to now, no general construction for a basis of the internal space  $\Pcal_-(X)$
  was known (cf.~\cite{ardila-postnikov-errata-2012, holtz-ron-2011, lenz-hzpi-2012}). The polynomials $f_z$ form such a basis.
  \begin{Corollary} 
  \label{Corollary:InternalPbasis}
   Let $X\subseteq \Lambda \subseteq U \cong \R^d$ be a list of vectors that is totally unimodular
     and spans $U$.
  Then  $\{ f_z : z \in \Zcal_-(X) \}$ is a basis for $\Pcal_-(X)$.
  \end{Corollary}
We also obtain a new basis for the central space $\Pcal(X)$. Let
 $w\in U$ be a \emph{short affine regular vector}, \ie a vector whose Euclidian length 
  is close to zero that is not contained in any hyperplane generated by sublists of $X$.
 Let 
   $\Zcal(X,w) := (Z(X) - w ) \cap \Lambda$ (see Figure~\ref{Figure:ZonotopeDelCon} for an example).
   It is known that $\dim \Pcal(X) = \abs{\Zcal(X,w)} = \vol(Z(X))$ \cite{holtz-ron-2011}.
  \begin{Corollary}
  \label{Corollary:newPbasis}
 Let $X\subseteq \Lambda \subseteq U\cong \R^d$ be a list of vectors that is totally unimodular and spans $U$. 
 Then
  $\{ f_z : z \in \Zcal(X,w)  \}$ is a basis for $\Pcal(X)$.
  \end{Corollary}
  This corollary will be used to prove the following new characterisation of the internal space $\Pcal_-(X)$.
  \begin{Corollary}
  \label{Corollary:interalPcontinuous}
  Let $X\subseteq \Lambda \subseteq U\cong \R^d$ be a list of vectors that is totally unimodular and spans $U$. 
  Then
   \begin{align}
    \Pcal_-(X) &= \{ f\in \Pcal(X) %
      : f(D) B_X \text{ is a continuous function} \}.
   \end{align}
  \end{Corollary}
\begin{Remark}
 There is a related result due to Dahmen-Micchelli (\cite{dahmen-micchelli-1985} or \cite[Theorem 13.21]{concini-procesi-book}): 
  for every function $f$ on %
   $\Zcal(X,w)$
  there exists a unique  
  function in $\DM(X)$ that agrees with $f$ on $\Zcal(X,w)$. Here,  $\DM(X)$ denotes the so-called discrete Dahmen-Micchelli space.
  The proof of the deconvolution formula in \cite{deconcini-procesi-vergne-2010c} relies on this result.
\end{Remark}
  
\subsection*{Organisation of the article}
The remainder of this article is organised as follows:
in Section~\ref{Section:ProofIdea}, we will first define deletion and contraction.
Then we will describe a method to make sense of derivatives of piecewise polynomial functions via limits and state a different version of the Main Theorem. 

In Section~\ref{Section:Examples} we will give some examples.
In Section~\ref{Section:ZonotopalAlgebra} we will recall some 
facts about zonotopal algebra, \ie about the space $\Pcal(X)$, the dual space $\Dcal(X)$
 and their connection with splines.
These will be  needed in the proof of the Main Theorem and its corollaries in Section~\ref{Section:TechnicalDetails}.
In %
 the appendix we will give an alternative proof of %
 the Main Theorem 
  in the univariate case that uses residues.  %

\subsection*{Acknowledgements}
The author would like to thank Mich\`ele Vergne for her comments on \cite{lenz-interpolation-2012} which lead to the present article.
 Some examples where calculated using the Sage Mathematics Software System \cite{sage-56}.

%
%
%
%

\section{Deletion-contraction, limits, and the extended Main Theorem}
\label{Section:ProofIdea}
In the first subsection, we will first define deletion and contraction and discuss the idea of 
the deletion-contraction proof of the Main Theorem. 
In the second subsection, 
 we will describe a method to make sense of derivatives of piecewise polynomial functions 
 via limits and then state a different version of the Main Theorem.

\subsection{Deletion and contraction}
Let $x\in X$. We call the list $X\setminus x$ the \emph{deletion} of $x$. 
 The image of $X\setminus x$ under the canonical projection $\pi_x : U \to U/\spa(x) =:U/x$ 
 is called the \emph{contraction} of $x$. It is denoted by $X/x$. %

The projection $\pi_x$ induces a map $ \sym(U)\to \sym(U/x)$ that we will also denote by $\pi_x$.
 If we identify $\sym(U)$ with the polynomial ring $\R[s_1,\ldots,s_d]$ 
  and $x=s_d$, then
  $\pi_x$ is the map from 
 $\R[s_1,\ldots, s_{d}]$ to $\R[s_1,\ldots, s_{d-1}]$ that sends $s_d$ to zero and $s_1,\ldots, s_{d-1}$ to themselves.
 The space $\Pcal(X/x)$ is contained in the symmetric algebra $\sym(U/x)$. 

 Note that since $X$ is totally unimodular,  $\Lambda/x \subseteq U/x$ is a lattice for every $x\in X$ and $X/x$ is  totally unimodular 
  with respect to this lattice. %

 Let $x\in X$. Using matroid theory terminology, we call $x$ a \emph{loop} if $x = 0$ and we call $x$ a \emph{coloop}
  if $\rank(X\setminus x) < \rank(X)$.
  \smallskip

 Recall that we defined $f_z = \psi_X(\todd(X,z))$ for $z \in \Zcal_-(X)$. 
By Theorem~\ref{Theorem:weakHoltzRon},
there is a unique polynomial $q_z^X = q_z\in \Pcal_-(X)$  \st
  $q_z (D) B_X |_\Lambda =\delta_{z}$  for every
  $z\in \Zcal_-(X)$. %
 In order to prove the Main Theorem it is sufficient to show that $f_z = q_z$.
 In fact,
  $q_z$ and $f_z$ behave in the same way under deletion and contraction:
 they both satisfy the equalities
 $x q_z^{X \setminus x} = q_z^X - q_{z+x}^X$ and $\pi_x(q_z^X) = q_{\bar z}^{X/x}$.
  Unfortunately, it is not obvious that $f_z\in \Pcal_-(X)$.
 Therefore, we have to make a detour.
Since
$\Pcal_-(X)$ is in general not spanned by polynomials of type $p_Y$ for some $Y\subseteq X$
(cf.\ \cite{ardila-postnikov-errata-2012}), it is quite difficult to handle this space.
The space $\Pcal(X)$ on the other hand has a basis which is very convenient for deletion-contraction
(cf.\ Proposition~\ref{Proposition:Pbasis}).
Therefore, we will 
 work with the larger space $\Pcal(X)$ and do a deletion-contraction proof there.
An extended version of the Main Theorem will be stated in the next subsection.
This will require some adjustments 
since for $f\in \Pcal(X)$, $f (D) B_X|_\Lambda$ 
 might not be well-defined.

\subsection{Differentiating piecewise polynomial functions and limits}

\begin{Definition}
 Let $H$ by a hyperplane spanned by a sublist  $Y \subseteq X$. A shift
  of such a hyperplane by a vector in the lattice $\Lambda$ is called an \emph{affine admissible hyperplane}.
   An \emph{alcove} is a connected component of the complement of the union of all affine admissible hyperplanes
   
 A vector $w\in U$ is called \emph{affine regular}, if it is not contained in any affine admissible hyperplane. %
 We call $w$ \emph{short affine regular} if its Euclidian length is close to zero.
\end{Definition}
  Note that on the closure of each alcove $\cfrak$,  $B_X$ agrees with a polynomial $p_\cfrak$.
 For example,  the six triangles in Figure~\ref{Figure:ThreeDirections} are the alcoves where $B_X$ agrees with a non-zero
  polynomial. 
 
 Fix a short affine regular vector  $w\in U$. %
   Let $u \in U$. Let $\cfrak \subseteq U$ be an alcove \st %
   $u$ and $u+ \eps w$ are contained in its closure for some small $\eps>0$ and 
  let $p_\cfrak$ be the polynomial that agrees with $B_X$ on the closure of $\cfrak$. 
  We define
   $\lim_w  B_X(z) := p_\cfrak(z)$ and for $f\in \sym(U)$
  \begin{align}
 \lim_w f(D_{\mathrm{pw}}) B_X (u) := f(D) p_\cfrak(u)
  \end{align}
 (pw stands for piecewise).
    Note that the limit can be dropped if  $f(D)B_X$ is continuous at $u$. Otherwise, the limit  is important: note for example that
  $\lim_w B_{(1)}(0)$ is either $1$ or $0$ depending on whether $w$ is positive or negative. 
  More information on this construction can be found in \cite{deconcini-procesi-vergne-2010c} where it was  introduced.
 We will later see that 
 $f_z (D)  B_X (D) $ is continuous if $z\in Z(X)\cap \Lambda$ is in the interior of $Z(X)$ and 
 discontinuous if it is on the boundary.
 
 Recall that $\Zcal(X,w) := (Z(X) - w ) \cap \Lambda$.
\begin{Theorem}
\label{Theorem:MainTheoremBoundary}
 Let $X\subseteq \Lambda\subseteq U \cong \R^d$ be a list of vectors that is totally unimodular and spans $U$.
Let $w$ be a short affine regular and let $z \in \Zcal(X,w)$. Then
 \begin{align}
\lim_{w}  f_z(D_{\mathrm{pw}}) B_X |_\Lambda = \lim_{w}  \todd(X,z)(D_{\mathrm{pw}}) B_X |_\Lambda = \delta_{z}. 
\end{align}
\end{Theorem}

%
%
%
%

\section{Examples}
\label{Section:Examples}
\begin{Example}

\begin{figure}[tb]
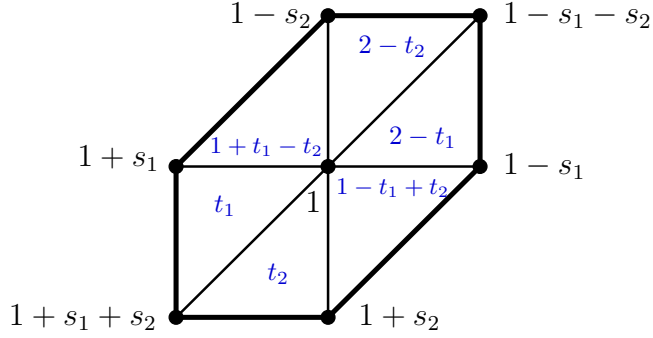
\caption{The box spline and the polynomials $f_z$ corresponding to Example~\ref{Example:ThreeDirections}.}
\label{Figure:ThreeDirections}
\end{figure}
We consider the three smallest one-dimensional examples. For $X=(1,1)$ we obtain 
 $\todd((1,1),1) = (1 + B_1 s + \ldots )(1 - B_1 s + \ldots ) = 1 + 0s + \ldots$
\\ Hence $f_1^{(1,1)}=1$. Furthermore,
%
%
\begin{align*}
 f_1^{(1,1,1)}  &=   1 + \frac s2,
&
f_2^{(1,1,1)} &= 
 1 - \frac{s}{2},
 \\
 f^{(1,1,1,1)}_1 &=  
   1 + s + \frac{s^2}{3},
 & f^{(1,1,1,1)}_2  &=  1 -   \frac{s^2}{6},
&  \text{and }
 f^{(1,1,1,1)}_3 =  1 - s + \frac{s^2}{3}.
\end{align*}
\end{Example}
\begin{Example}
\label{Example:ThreeDirections}
 Let $ X = ( (1,0), (0,1), (1,1) ) \subseteq \Z^2 $. Then $ \Pcal_-(X) = \R $, $\Pcal(X) = \spa\{1,s_1,s_2\} $,
  $\Zcal_-(X) = \{ (1,1) \} $,
 and $f_{(1,1)} = 1$.
 $\Pi_X(u_1,u_2) \cong [0, \min(u_1,u_2)]\subseteq \R^1$.
 The multivariate spline and the vector partition function are:
 \begin{align}
  T_X( u_1,u_2) &= \begin{cases}
               u_2 & \text{for } 0\le u_2 \le u_1 \\
               u_1 & \text{for } 0\le u_1 \le u_2 \\
              \end{cases}
&
 \text{and } \Tcal_X( x,y) &= \begin{cases}
               u_2+1 & \text{for }  0 \le  u_2 \le u_1 \\
               u_1+1 & \text{for }  0 \le  u_1 \le u_2 \\
              \end{cases}.
 \end{align}
 Corollary~\ref{Corollary:KhovanskiiPukhlikov}  correctly predicts that $T_X(u)|_{\Z^2} = \Tcal_X(u - (1,1))$. 
  Figure~\ref{Figure:ThreeDirections} shows the six non-zero local pieces of $B_X$ 
  and the seven polynomials $f_z$ attached to the lattice points of the zonotope $Z(X)$.
\end{Example}

\begin{figure}[t]
 \begin{center}
 \includegraphics{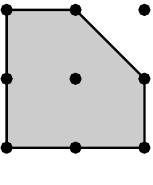}
 \hspace*{2cm}
 \includegraphics{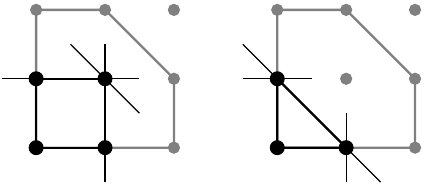}
\end{center}
\caption{The pentagon $\pentagon$, square $\square$, and triangle $\Delta$ that we discuss in Example~\ref{Example:SmoothPolygon}.}
\label{Figure:SmoothPolygon}
\end{figure}
 \begin{Example}
 \label{Example:SmoothPolygon}
We consider the polygons in Figure~\ref{Figure:SmoothPolygon}, which we will denote by $\pentagon$, $\square$, and $\Delta$. These polytopes are defined by the matrix
\begin{align}
 X = \begin{bmatrix}
      1 & 0 & 0 & 1 & 0 \\
      0 & 1 & 0 & 0 & 1 \\
      0 & 0 & 1 & 1 & 1 
     \end{bmatrix},
\end{align}
where the first three columns correspond to slack variables and the rows of the last two are the normal vectors of the movable facets.
 The corresponding zonotope has two interior lattice points:
 $(1, 1, 1)$ corresponds to $\square$ and  $(1, 1, 2)$ to $\Delta$.
The projections of the Todd operators  are $ f_{\square} = 1 + s_3/2$ and $f_{\Delta} = 1 -  s_3/2$, where $s_3$ corresponds to shifting the diagonal face of the pentagon.
 Elementary calculations show 
 that shifting this face outwards\footnote{This means replacing the inequality $a+b \le 3$ by $a+b \le 3+ \eps$.}
 by $\eps \in [-1,1]$ increases the volume of the pentagon by
$(\eps - \frac 12\eps^2)$. 
This implies
 $\diff{s_3}\vol(\text{pentagon})= 1$.

 The volume of the pentagon is $3.5$.
 Corollary~\ref{Corollary:KhovanskiiPukhlikov}  correctly predicts that
 \begin{align}
  \abs{\square \cap \Z^2} &= \vol(\pentagon) + \frac 12 \diff{ s_3} \vol(\pentagon) = 4 \\
  \text{ and }  
  \abs{\Delta \cap \Z^2} &= \vol(\pentagon) - \frac 12 \diff{ s_3} \vol(\pentagon) = 3.
 \end{align}

 The projection to $\Pcal(X)$ of the unshifted Todd operator $\todd(X,0)$ is a lot more complicated:
 $ f_{\pentagon}  = 1 + s_{1} + s_{2} + \frac{3}{2} s_{3} +  s_{1} s_{2} + s_{1} s_{3} + s_{2} s_{3} + s_{3}^{2} $.

On the other hand, since $\Zcal_-(X)$ contains only two points, the box spline $B_X$ must assume the value $\frac 12$ at both 
 points. Then \eqref{eq:DahmenMicchelli} correctly predicts that the volume of the pentagon is the arithmetic mean of the number of integer points
  in the square and the triangle.

 \end{Example}
  \begin{Remark}
  If $X$ is not totally unimodular, then $\psi_X( \todd(X,z))$ is in general not contained in $\Pcal_-(X)$.
  Consider for example $X = (2,1)$ and $\todd(X,z) = \frac{2x}{e^{2x}-1}\frac{x}{1-e^{-x}}$.
  Then
  \begin{align}
    \psi_X(f_z) = \psi_X( 2(1+ 2B_1x)(1-B_1x)) =  2 - x \not\in \Pcal_-(X)= \R.
  \end{align}

 \end{Remark}
 \begin{figure}[tb]
\begin{center}
 \includegraphics[width=8cm,
 ]{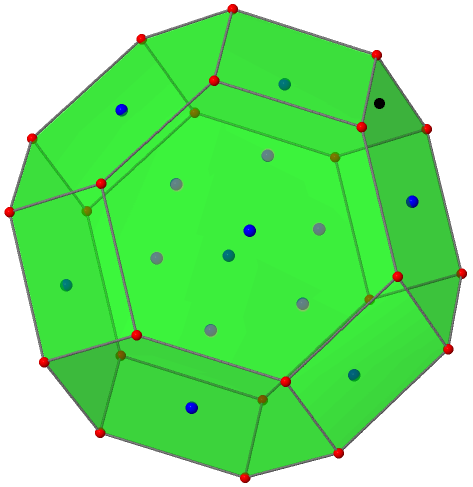}
\end{center}
\caption{The zonotope corresponding to Example~\ref{Example:Kfour}. It has six interior lattice points.}
\label{Figure:KfourZonotope}
\end{figure}
 \begin{Example}
 \label{Example:Kfour}
Let $X$ be a reduced oriented incidence matrix of the complete graph on $4$ vertices / the set of 
 positive roots of the root system $A_3$ (cf.\ Figure \ref{Figure:KfourZonotope}), \ie
\begin{align}
 X &=
 \left(\begin{array}{rrrrrr}
1 & 0 & 0 & 1 & 1 & 0 \\
0 & 1 & 0 & -1 & 0 & 1 \\
0 & 0 & 1 & 0 & -1 & -1
\end{array}\right). \displaybreak[2]\\
f_{ (1, 1, 0) } &=  -\frac{1}{6} t_{1} t_{2} - \frac{1}{6} t_{1} t_{3} + \frac{1}{3} t_{2} t_{3} + \frac{1}{2} t_{1} - \frac{1}{2} t_{2} - \frac{1}{2} t_{3} + 1 \\
f_{ (1, 1, -1) } &=  -\frac{1}{6} t_{1} t_{2} + \frac{1}{3} t_{1} t_{3} - \frac{1}{6} t_{2} t_{3} + \frac{1}{2} t_{1} - \frac{1}{2} t_{2} + \frac{1}{2} t_{3} + 1 \\
f_{ (2, 0, -1) } &=  -\frac{1}{6} t_{1} t_{2} - \frac{1}{6} t_{1} t_{3} + \frac{1}{3} t_{2} t_{3} - \frac{1}{2} t_{1} + \frac{1}{2} t_{2} + \frac{1}{2} t_{3} + 1 \\
f_{ (1, 0, 0) } &=  \frac{1}{3} t_{1} t_{2} - \frac{1}{6} t_{1} t_{3} - \frac{1}{6} t_{2} t_{3} + \frac{1}{2} t_{1} + \frac{1}{2} t_{2} - \frac{1}{2} t_{3} + 1 \\
f_{ (2, 1, -1) } &=  \frac{1}{3} t_{1} t_{2} - \frac{1}{6} t_{1} t_{3} - \frac{1}{6} t_{2} t_{3} - \frac{1}{2} t_{1} - \frac{1}{2} t_{2} + \frac{1}{2} t_{3} + 1 \\
f_{ (2, 0, 0) } &=  -\frac{1}{6} t_{1} t_{2} + \frac{1}{3} t_{1} t_{3} - \frac{1}{6} t_{2} t_{3} - \frac{1}{2} t_{1} + \frac{1}{2} t_{2} - \frac{1}{2} t_{3} + 1 
\end{align}
\end{Example}

%
%
%
%

\section{Zonotopal Algebra}
\label{Section:ZonotopalAlgebra}
 In this section we will recall a few facts about the space $\Pcal(X)$ and its dual, the space $\Dcal(X)$
 that is spanned by the local pieces of the box spline.
 The theory around these spaces was named zonotopal algebra by 
 Holtz and Ron in \cite{holtz-ron-2011}.
 This theory  allows us to explicitly describe the map $\psi_X$. This description will then be used to describe the behaviour of the polynomials $f_z$ under deletion and contraction.
 This section is an extremely condensed version of \cite{lenz-forward-2012}.

Recall that the list of vectors $X$ is contained in a vector space $U \cong \R^d$ and that we denote the dual space by $V$.
We start by defining a pairing between the symmetric algebras $\sym(U) \cong \R[s_1,\ldots, s_d]$ and 
$\sym(V) \cong \R[t_1,\ldots, t_d]$:
\begin{align}
\begin{split}
 \pair{\cdot}{\cdot} : \R[s_1,\ldots, s_d] \times \R[t_1,\ldots, t_d] &\to \R  \\
    \pair{p}{f} := \left( p\left(\diff{t_1},\ldots, \frac{\partial}{\partial t_d} \right)  f \right) (0), 
\end{split}
\end{align}
\ie we let $p$ act on $f$ as a differential operator and take the degree zero part of the result.
Note that this pairing  extends to a pairing
 $\pair{\cdot}{\cdot} : \R[[s_1,\ldots, s_d]] \times \R[t_1,\ldots, t_d] \to \R$.

  It is known that the box spline $B_X$ agrees with a polynomial on the closure of each alcove.
  These local pieces and their partial derivatives
  span the \emph{Dahmen-Micchelli} space $\Dcal(X)$. 
  This space can be described as the kernel of the cocircuit ideal $\Jcal(X)$, namely
  \begin{align}
     \label{eq:DahmenMicchelliCocircuit}
     \Dcal(X) = \{  f\in \sym(V) :  \pair{p}{f} = 0 \text{ for all } p \in \Jcal(X)  \}.
  \end{align}

 We will now explain a construction of certain polynomials that is essentially due to
 De Concini, Procesi, and Vergne \cite{deconcini-procesi-vergne-2010b}.
  Let $Z\subseteq U$ be a finite list of vectors and let $B=(b_1,\ldots, b_d) \subseteq Z$ be a basis. 
 It is important that the basis is ordered and that this order is
  the order obtained by restricting
 the order on $Z$ to $B$. For $i\in\{ 0,\ldots, d\}$, we define $S_i = S_i^B := \spa \{ b_1,\ldots, b_i \}$. 
Hence
 \begin{align}
 \label{eq:flagofsubspaces}
   \{ 0 \} = S_0^B  \subsetneq S_1^B \subsetneq S_2^B \subsetneq \ldots \subsetneq S_d^B = U \cong \R^d
 \end{align}
is a flag of subspaces.
 Let $u\in S_{i}\setminus S_{i-1}$.  
 The vector $u$ can be written as $u=\sum_{\nu=1}^i \lambda_\nu b_\nu$ in a unique way. Note that $\lambda_i\neq 0$.
 If $\lambda_i>0$, we call $u$ positive and if $\lambda_i< 0 $, we call $u$ negative.
We partition $Z \cap (S_{i} \setminus S_{i-1})$ as follows:
 \begin{align}
   P_{i}^B &:= \{  u \in Z \cap (S_{i} \setminus S_{i-1}) : u \text{ positive}\} \\  
 \text{and }   N_{i}^B &:= \{ u \in Z \cap (S_{i} \setminus S_{i-1}) : u \text{ negative} \}.
\intertext{We define}
  T_{i}^{B+}  &:= (-1)^{\abs{N_{i}}} \cdot T_{P_i} * T_{-N_i} 
 \text{ and } T_{i}^{B-}:= (-1)^{\abs{P_i}} \cdot T_{-P_i} * T_{N_i}.
 \end{align}
Note that  $T_{i}^{B+}$ is supported in  $\cone(P_i,-N_i)$ and 
 that 
 \begin{align}
 T_{i}^{B-}(x) = (-1)^{\abs{P_i\cup N_i}} T_i^{B+}(- x). 
 \end{align}
Now define 
\begin{align}
   R_{i}^{B}  := T_i^{B+} -  T_i^{B-} %
 \qquad \text{ and } \qquad
 R^B_Z = R^B := R_1^B * \cdots * R_d^B.
\label{eq:BasisElementConvolution}
\end{align}

 We denote the set of all sublists $B\subseteq X$ that are bases for $U$ by $\BB(X)$.
Fix a basis $B\in \BB(X)$. A vector $x\in X\setminus B$ is called \emph{externally active}
 if $x\in \spa \{ b \in B : b \le x \}$,
\ie $x$ is the maximal element of the unique circuit contained in $B\cup x$.
 The set of all externally active elements is denoted $E(B)$.
 
 \begin{Definition}[Basis for $\Dcal(X)$]
\Xintro
We define
\begin{align}
 \Bcyr(X) :=  \{ \abs{\det(B)}  R^B_{X\setminus E(B)} : B\in \BB(X)   \}.
\end{align}
\end{Definition}

\begin{Proposition}[\cite{dyn-ron-1990,jia-1990}] %
\label{Proposition:PDduality}
\Xintro
Then the spaces $\Pcal(X)$ and $\Dcal(X)$ are dual under the pairing $\pair{\cdot}{\cdot}$, \ie
\begin{align}
 \begin{split}
 \Dcal(X) &\to \Pcal(X)^* \\
  f &\mapsto \pair{\cdot}{f}
 \end{split}
\end{align}
is an isomorphism.
\end{Proposition}

\begin{Proposition}[\cite{dyn-ron-1990}%
 ]
\label{Proposition:Pbasis}
\Xintro
  A basis for $\Pcal(X)$ is given by  
 \begin{align}
  \Bcal(X) &:= \{ Q_B   : B \in \BB(X) \},
 \end{align}
 where $Q_B:=p_{ X\setminus (B \cup E(B))}$.
\end{Proposition}

\begin{Theorem}[\cite{lenz-forward-2012}]
\label{Theorem:MainTheoremCentral}
\Xintro
 Then
$\Bcyr(X)$ %
is a basis for $\Dcal(X)$ and this 
 basis is dual to the basis $\Bcal(X)$ for the central $\Pcal$-space  $\Pcal(X)$. %
\end{Theorem}

\begin{Remark}
\label{Remark:ProjectionInBasis}
Theorem~\ref{Theorem:MainTheoremCentral} yields an 
 explicit formula for the projection map $\psi_X : \R[[s_1,\ldots, s_d]] \to \Pcal(X)$ that we have 
 defined on page \pageref{eq:ProjectionMap}:
\begin{align}
  f_z := \psi_{X} (\todd(X,z)) =  \sum_{B\in \BB(X)} \pair{\todd(X,z)}{R^B}\, Q_B.
\end{align}
\end{Remark}

%
%
%
%

\section{Proofs} %
\label{Section:TechnicalDetails}
  In this section we will prove the Main Theorem and its corollaries. %
 The proof uses a deletion-contraction argument.
 Deletion-contraction identities for the polynomials $f_z$ will be obtained 
 based on the following idea:
one can write $\psi_X(f)$ as
\begin{align}
\label{eq:DelConDecomposition}
  \psi_X(f) &= \sum_{B\in \BB(X)} \pair{f}{R_B}\, Q_B 
  =
  \sum_{x\not \in B\in \BB(X)} \pair{f}{R_B}\, Q_B
   +
  \sum_{x \in B\in \BB(X)} \pair{f}{R_B}\, Q_B
  \end{align}
(cf.\ Remark~\ref{Remark:ProjectionInBasis}).
We will see that the first sum  on the right-hand side of \eqref{eq:DelConDecomposition} corresponds to $\Pcal_-(X\setminus x)$ and the second to $\Pcal_-(X/x)$. %
Note that $\psi_X(f)$ is by definition independent of the order imposed on the list $X$, 
 while each of the two sums depends on this order.

  It is an important observation that for $x\in X$ and $z\in \Lambda$
\begin{align}
\label{eq:fmultx}
 x \cdot \todd(X\setminus x,z) = \todd(X, z) - \todd(X, z+x) 
\end{align}
  holds because $\frac{x}{1 - e^{-x}}(1 - e^{-x})=x$.

\begin{Lemma}
\label{Lemma:Injection}
Let $x\in X$ and let $f\in \sym(U)$. Then
\begin{align}
  x\psi_{X\setminus x}(f) = %
  \psi_X(xf).
 \end{align}
Furthermore, for all $z\in \Lambda$,
 \begin{align}
  x f_z^{X\setminus x} =  f_{z}^X - f_{z+x}^X. 
\end{align}
\end{Lemma}
\begin{proof}
 Note that the statement is trivial for $x=0$, so from now on we assume that $x$ is not a loop.
 Since $\psi_X(f)$ is independent of the order imposed on $X$, we may rearrange the list elements \st $x$ is minimal.

 Let $x\in B\in \BB(X)$.  Since $x$ is minimal, $\spa(x) \cap (X\setminus E(B)) = \{x\}$. By  
Lemma~4.5 in \cite{lenz-forward-2012}, this implies that $\pair{xf}{R_B}=  \pair{f}{D_xR_B} = \pair{f}{0} = 0$.

Let $x\not\in B \in \BB(X)$. Since $x$ is minimal, this implies that $x\not\in E(B)$. Then  
 \begin{align}
  \pair{xf}{R^B_{X\setminus E(B)}} &= \pair{f}{D_x R^B_{X\setminus E(B)}}    
  = \pair{f}{R^B_{X\setminus (E(B)\cup x)}}.
 \end{align}
 Using the two previous observations we obtain:
\begin{align}
  \psi_X(xf)  &=  \underbrace{\sum_{x \in B \in \BB(X)} \pair{xf}{R^B_{X\setminus E(B)}} \, Q_B}_{ = 0 } 
  + \sum_{ x \not \in B \in \BB(X) } \pair{xf}{R^B_{X\setminus E(B)}} \, Q_B
  \\
  &=  \sum_{  B \in \BB(X \setminus x)}  \pair{f}{R^B_{X\setminus (E(B)\cup x)}}\,  x\, Q_B^{X\setminus x} = x\psi_{X\setminus x}(f).
\end{align}
The second to last equality follows from the fact that $Q_B^X = p_{X\setminus (B\cup E(B))} = xp_{X\setminus (B\cup E(B) \cup x)} 
  = Q_B^{X\setminus x}$ if $x\not\in B\cup E(B)$.

 We can deduce the second claim  using \eqref{eq:fmultx}: 
 \begin{align}
  x f_z^{X \setminus x} &= 
  x \psi_{X\setminus x}( \todd ( X\setminus x, z )) = \psi_X( x \todd(X\setminus x, z) ) 
  \\ &=
  \psi_X(  \todd(X, z)  -  \todd (X, z+x) ) = f_{z}^X - f_{z+x}^X. \qedhere 
 \end{align}
\end{proof}
Recall that
$\pi_x : \sym(U) \to \sym(U/x)$ denotes the canonical projection.
\begin{Lemma}
 \label{Lemma:Surjection}
Let $x\in X$ be neither a loop nor a coloop. 
 Let $z\in \Lambda$ and let $\bar z = \pi_x(z)\in \Lambda/x$.
Then
 \begin{align}
    \pi_x(f_z) = f_{\bar z}.
 \end{align}
\end{Lemma}
\begin{proof}
Since the maps $\pi_x$ and $\psi_X$ are independent of the order imposed on $X$, we may rearrange the list elements \st $x$ is minimal.

 Let $x\in B\in \BB(X)$. Given that $x$ is minimal,
 $R^B = (T_x - T_{-x}) *R^{B\setminus x}_{X\setminus (E(B)\cup x)}$ follows. Hence
  $R^B$ is constant in direction $x$ so we can interpret it as a function on  $U/x$ and identify it with $R^{\bar B}$. 
  In $\todd(X,z)$, the factor $x/(1- e^{-x})$  becomes $1$ if we set $x$ to $0$.
   Note that $x$ divides $Q_B$ if $x\not\in B$ since $x$ is minimal. Then
 \begin{align}
  \begin{split} \pi_x (f_z) &=   
  \underbrace{  \pi_x \Biggl(\sum_{x\not \in B\in \BB(X )}  \pair{\todd(X,z)}{R^B}\, Q_B \Biggr) }_{=0, \text{ since } x|Q_B.}
  \\ & \qquad+  \pi_x \Biggl(\sum_{x  \in B\in \BB(X )}  \pair{\todd(X,z)}{R^B}\, Q_B \Biggr)
  \end{split}
 \\ &=
        \sum_{\bar B\in \BB(X/x )}  \pair{ \todd( X/x, \bar z) }{R^{\bar B}}\, Q_{\bar B} =  f_{\bar z}. %
\qedhere
        \end{align}
\end{proof}

 In \cite{lenz-interpolation-2012}, the author showed that  the function
\begin{align}
 \gamma_X : \Pcal_-(X) \to  \Xi(X):= \{ f : \Lambda \to \R : \supp(f) \subseteq \Zcal_-(X)  \} 
\end{align}
 that maps $p$ to  $p (D) B_X|_\Lambda$ is an isomorphism.
We will now extend $\gamma_X$ to a map that is an isomorphism between $\Pcal(X)$ 
and a superspace of $\Xi(X)$. 

For a short affine regular vector $w\in U$ we define
\begin{align}
\begin{split}
\gamma_X^w : \Pcal(X)&\to \Xi^w(X):= \{ f : \Lambda \to \R : \supp(f) \subseteq  \Zcal(X,w)  \}  %
    \\
	      p &\mapsto \lim_w p(D_{\mathrm{pw}}) B_X|_\Lambda .   
\end{split}
	      \end{align}
\begin{figure}[tb]
\begin{center}
 \includegraphics{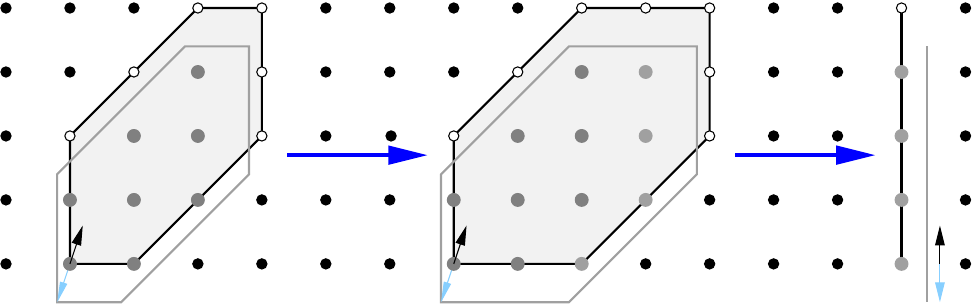}
  \caption{Deletion and contraction for the set $\Zcal(X,w)$.}
  \label{Figure:ZonotopeDelCon}
\end{center}
\end{figure}
\begin{Proposition}
\label{Proposition:ExactSequences}
 Let $X\subseteq \Lambda\subseteq U \cong \R^d$ be a list of vectors that is totally unimodular and spans $U$.
  Let $x\in X$ be neither a loop nor a coloop.
  
  Then
 the following diagram of real vector spaces is 
 commutative, the rows are exact and the vertical maps are isomorphisms:
\begin{align}
\label{eq:ExactSequencesExtended}
\xymatrix{
  0 \ar[r] &     \Pcal(X\setminus x) \ar[r]^{\cdot x} \ar[d]^{\gamma^w_{X\setminus x}}  
   & \Pcal(X) \ar[r]^ { \pi_x } \ar[d]^{\gamma^w_{X}}
  &  \Pcal(X/x )\ar[r] \ar[d]^{\gamma^{\bar w}_{X / x}}  &  0
 \\
 0 \ar[r] & \Xi^w(X\setminus x) \ar[r]^{\nabla_x} &
     \Xi^w(X) \ar[r]^{\Sigma_x} &  \Xi^{\bar w}(X/x) \ar[r]  &  0
  }
\end{align}
\begin{align}
\text{where } \nabla_x(f)(z) &:=  f(z) - f(z - x) \nonumber \\
 \text{and }  \Sigma_x (f) (\bar z) &:=  \sum_{x\in\bar z \cap \Lambda } f(x)  = \sum_{\lambda\in \Z} f( \lambda x + z) \text{ for some } z\in \bar z.
\nonumber
\end{align}
\end{Proposition}
\begin{proof}
 First note that the map
 $\Zcal(X,w) \setminus \Zcal(X\setminus x, w) \to \Zcal(X/x, \bar w)$
 that sends $z$ to $\bar z$ is a bijection. 
 This is a variant of  \cite[Lemma 17]{lenz-interpolation-2012} and it can be proved in the same way, using
 the fact that 
 $\dim \Pcal(X) = \vol(Z(X))= \abs{\Zcal(X,w)}$, which was established in \cite{holtz-ron-2011}.
 See also Figure~\ref{Figure:ZonotopeDelCon}.
 This  implies  that the second row is exact. Exactness of the first row is known (\eg \cite[Proposition 4.4]{ardila-postnikov-2009}).

 If $X$ contains only loops and coloops, then $\Pcal(X)=\R$ and $B_X$ is the indicator function of 
 the parallelepiped spanned by the coloops.
 For every short affine regular $w$, $\Zcal(X,w)$ contains a unique point $z_w$ and
 $\lim_w B_X(z_w) = 1$. Hence $\gamma_X^w$ is an isomorphism in this case.
 Note that this also holds for $U= \Lambda = \{0\}$. 
 In this case $w=0$, $Z(X)= \{0\} = \Zcal(X,w)$, and $B_X=\lim_w B_X = 1 $.

  The rest of the proof is analogous to the proof of 
 \cite[Proposition]{lenz-interpolation-2012} and therefore omitted. 
\end{proof}

For $z\in \Zcal(X,w)$, let
  $q_z^{X,w} = q_z^w := (\gamma_X^w)^{-1}( \delta_z ) \in \Pcal(X)$. By construction, this polynomial  satisfies
  $\lim_w q_z^w(D_{\mathrm{pw}}) B_X|_\Lambda = \delta_z$.
\begin{Lemma}
 \label{Lemma:qequalsf}
 Let $w \in U$ be a short affine regular vector and let  $z\in \Zcal(X,w)$.
 Then $q_z^w = f_z$.

 In particular,
      $q_z^w = q_z^{w'}$ for short affine regular  vectors $w$ and $w'$ \st $z\in \Zcal(X,w) \cap \Zcal(X,w')$.
\end{Lemma}
\begin{proof}
 Fix a short affine regular $w$. 
 We will show by induction that $q_z^w = f_z$ for all $ z \in \Zcal(X,w) $.

 If $X$ contains only loops and coloops, then $\Pcal(X)=\R$ and $B_X$ is the indicator function of the parallelepiped spanned by the coloops.
 Hence $f_z= \psi_X(1 + \ldots) = 1 =  q_z^w $.
 Note that this is also holds for $U= \Lambda = \{0\}$. In this case $w=0$, $Z(X)= \{0\}= \Zcal(X,w)$, and $B_X=\lim_w B_X = 1$.

 Now
 suppose that there is an element $ x\in X $ that is neither a loop nor a coloop and suppose that 
  the statement is already true for $X/x$ and $X\setminus x$. 

   Let $ g_z^w = g_z := f_z - q_z^w \in \Pcal(X)$.

 \emph{Step 1. If $z' = z+\lambda x$ for some $\lambda\in \Z$, then  $g_z = g_{z'}$:}
 by Lemma~\ref{Lemma:Injection}, $xf_z^{X\setminus x} = f^X_z - f^X_{z+x}$.
Because of the commutativity of \eqref{eq:ExactSequencesExtended}, %
 $\gamma_X^w(xq_z^{X\setminus x,w}) = \nabla_x ( \delta_z )$. Since $\gamma_X^w$ is an isomorphism this implies $xq_z^{X\setminus x,w} = q_z^{X,w} - q_{z+x}^{X,w}$.
  By assumption, $q_z^{X\setminus x} = f_z^{X\setminus x}$. Hence we can deduce that
  $f^X_z - f^X_{z+x} = q_z^{X,w} - q_{z+x}^{X,w}$ and the claim follows.

\emph{Step 2. $\pi_x(g_z)=0$:}
 it follows from 
 Proposition~\ref{Proposition:ExactSequences}
  that $\pi_x(q_z^w)=q_{\bar z}^{\bar w}$.
Taking into account Lemma~\ref{Lemma:Surjection} this implies
 $\pi_x( g_z ) = f_{\bar z} - q_{\bar z}^{\bar w}$. This is zero by assumption.

\emph{ Step 3. $ g_z = 0 $: }
 using %
 the commutativity of \eqref{eq:ExactSequencesExtended} again
and steps 1 and 2, we obtain
\begin{align}
    \abs{ \Zcal( X,w ) \cap ( \spa(x)+z) } \cdot \gamma_X^w(g_z) =  \Sigma_x ( \gamma_X^w( g_z) ) = \gamma^{\bar w}_{X/x} ( \pi_x ( g_z))= 0.
\end{align}
 Hence $\gamma_X^w(g_z)=0$
and the claim follows since $\gamma_X^w$ is an isomorphism.
\end{proof}
\begin{Lemma}
 \label{Lemma:BoxSplineDiffPspace}
 Let $f\in \R[[s_1,\ldots, s_d]]$ and let $w$ be a short affine regular vector. 
 \begin{asparaenum}
 \item Then  $ \lim_w f(D_{\mathrm{pw}})B_X = \lim_w \psi_X( f )(D_{\mathrm{pw}})B_X $.
 \item %
 If $\psi_X(f)(D)B_X $ is continuous, then 
  $f(D)B_X$ can be extended continuously \st $\psi_X(f)(D)B_X = f(D)B_X $. 
 \end{asparaenum}
\end{Lemma}
\begin{proof}
 Let $u\in U$ and let $\cfrak$ be the alcove \st 
  $u+ \eps w\in \cfrak$ for all sufficiently small $\eps> 0$.
 Let $p_\cfrak$ be the polynomial that agrees with $B_X$ on the closure of $\cfrak$.
 By definition of the map $\psi_X$,  
 for all $i$ the degree $i$ part of 
 $j = f  - \psi_X(f)  $ is contained in $ \Jcal(X) $.
  By definition of $\Dcal(X)$, $p_\cfrak\in \Dcal(X)$.  Then \eqref{eq:DahmenMicchelliCocircuit}
  implies $j (D) p_\cfrak = 0$.
   Hence $\lim_w f(D_{\mathrm{pw}}) B_X (u) = f(D) p_\cfrak (u) = 
    \psi_X(f) (D) p_\cfrak(u) = \lim_w \psi_X(f) (D_{\mathrm{pw}}) B_X (u)  $. %

    If $\psi_X(f)(D)B_X$ is continuous, then 
    \begin{align}
    \psi_X(f)(D)B_X(u) = \lim_w \psi_X(f)(D_{\mathrm{pw}}) B_X(u) = \lim_w f(D_{\mathrm{pw}})
     B_X(u)
     \end{align}
     for all short affine regular vectors $w$, so if we define  $f(D)B_X (u) := \psi_X(f)B_X(u)$ for all $u\in U$,
     we obtain a continuous extension of  $f(D)B_X$ to $U$. 
   \end{proof}

\begin{proof}[Proof of Theorem~\ref{Theorem:MainTheorem} (Main Theorem) and of Theorem~\ref{Theorem:MainTheoremBoundary}]
 If $w$ is a short affine regular vector and $z\in \Zcal(X,w)$, then
  $\lim_w f_z (D_{\mathrm{pw}}) B_X|_\Lambda = \delta_z$ 
 by Proposition~\ref{Proposition:ExactSequences} and Lemma~\ref{Lemma:qequalsf}. %

 Now let $z \in \Zcal_-(X)$. Then by 
 Theorem~\ref{Theorem:weakHoltzRon}, there exists $q_z^-\in \Pcal_-(X)$ \st
  $q_z^-(D)B_X|_{\Zcal_-(X)} = \delta_z$ and $q_z^-(D)B_X$ is continuous.
 Continuity implies that  $q_z^-(D)B_X$ vanishes on the boundary of $Z(X)$, so $q_z^-(D)B_X|_\Lambda = \delta_z$.
  Furthermore, for any short affine regular $w$, $\lim_w q_z^- (D_{\mathrm{pw}}) B_X|_\Lambda = \delta_z$.
   Since $\gamma_X^w$ is injective, $q_z^- = q_z^w = f_z$ must hold. Hence, $f_z\in \Pcal_-(X)$.

 
 To finish the proof, note that 
 Lemma~\ref{Lemma:BoxSplineDiffPspace} implies that 
$ \lim_w f_z(D_{\mathrm{pw}})B_X|_\Lambda = \lim_w \todd(X,z)(D_{\mathrm{pw}})B_X |_\Lambda$ %
and if $f_z(D) B_X$ is continuous it is possible to extend $\todd(X,z)(D) B_X$ continuously \st
 $f_z(D) B_X = \todd(X,z)(D) B_X$.
\end{proof}

\begin{proof}[Proof of Corollary~\ref{Corollary:DahmenMicchelli}]
Let $ p := \sum_{z \in \Zcal_-(X)} B_X(z) f_z$.  By the Main Theorem,
  $p\in \Pcal_-(X)$ and $p (D) B_X$ and $B_X = 1 (D)  B_X$ agree on $\Zcal_-(X)$. %
   By the uniqueness part of Theorem~\ref{Theorem:weakHoltzRon} this implies $p = 1$.

   Using the previous observation and Corollary~\ref{Corollary:KhovanskiiPukhlikov} we can deduce  \eqref{eq:DahmenMicchelli}:  
  \begin{align}
  B_X *_d \Tcal_X(u) &= \sum_{\lambda \in \Lambda} B_X( u - \lambda ) \Tcal_X(\lambda) = \sum_{ z \in \Lambda } B_X( z ) \Tcal_X( u - z)
  \\ &=
  \sum_{ z \in \Lambda } B_X( z )  f_z 
  (D) T_X(u) = T_X(u). \qedhere
  \end{align}
\end{proof}

\begin{proof}[Proof of Corollary~\ref{Corollary:InternalPbasis}]
 By Theorem~\ref{Theorem:MainTheorem}, $\{ f_z : z\in \Zcal_-(X) \}$ is a linearly independent subset of $\Pcal_-(X)$.
 This set is actually a basis since
 $\abs{\Zcal_-(X)} = \dim \Pcal_-(X)$.
 This equality follows from Theorem~\ref{Theorem:weakHoltzRon}. It was  also proven  in \cite{holtz-ron-2011}.
\end{proof}
\begin{proof}[Proof of Corollary~\ref{Corollary:newPbasis}] 
 By construction, $\{f_z : z\in \Zcal(X,w) \} \subseteq \Pcal(X)$.
 Proposition~\ref{Proposition:ExactSequences} and Lemma~\ref{Lemma:qequalsf} imply that the set  is actually a basis. 
\end{proof}
\begin{proof}[Proof of Corollary~\ref{Corollary:interalPcontinuous}]
 ``$\subseteq$'' is part of Theorem~\ref{Theorem:weakHoltzRon}.

``$\supseteq$'': Let $p\in \Pcal(X)$ and suppose that $p(D)B_X$ is continuous. Let $w$ be a short affine regular vector.
 By Corollary~\ref{Corollary:newPbasis}, there exist uniquely determined $\lambda_z\in \R$ \st
 $p = \sum_{z\in \Zcal(X,w)} \lambda_z f_z$.

 Let $z\in \Zcal(X,w) \setminus \Zcal_-(X)$.
 Since 
 $z-w\not\in Z(X)$, $\lim_{-w} p(D_{\mathrm{pw}}) B_X (z) = 0$ holds.
 As we assumed that $p(D)B_X$ is continuous,
 this implies that 
 \begin{align}
  0 = \lim_w p(D_{\mathrm{pw}}) B_X (z) = \sum_{y\in \Zcal(X,w) } \lambda_y \delta_y(z) = \lambda_z.
  \end{align}
 Hence $p = \sum_{z\in \Zcal_-(X)} \lambda_z f_z$, which is in $\Pcal_-(X)$ by the Main Theorem.
\end{proof}

%
%
%
%

\begin{appendix}

\section{The univariate case and residues}
\label{Section:UnivariateCase}

In this appendix we will  give an alternative proof of (a part of) the Main Theorem in the univariate case. 
 This proof was provided by Mich\`ele Vergne.
 
A totally unimodular list of vectors $X\subseteq \Z^1 \subseteq \R^1$ contains only entries in $\{-1,0,1\}$.
Suppose that it contains  $a$ times $-1$ and $b$ times $1$. The zonotope $Z(X)$ is then the interval
 $[- a, b]$. We may assume that $X$ does not contain any zeroes and we choose $N$ \st  $a + b = N + 1$.
 Then
 $\Pcal(X)= \R[s]_{\le N}$ and $\Pcal_-(X)= \R[s]_{\le N-1}$.
 
Let $z \in \Zcal_-(X) = \{-a +1,\ldots, b -1   \}$. Then
there exist  \emph{positive} integers $\alpha$ and $\beta$ \st
\begin{align}
\todd(X,z):=e^{-zs}\left(\frac{s}{e^{s} - 1}\right)^{a} \left(\frac{s}{1 - e^{-s}}\right)^{b} 
 = \left(\frac{s}{e^{s} - 1}\right)^{\alpha} \left(\frac{s}{1 - e^{-s}}\right)^{\beta}.
\end{align}

In this case, $\psi_X : \R[s] \to \R[s]_{\le N}$ is the map that forgets all monomials of degree greater or equal $N+1$. 
\begin{Lemma}
\label{Proposition:OneDimCase}
 Let $X\subseteq \Z^1\subseteq \R^1$ be a list of vectors that is totally unimodular and
  let $z$ be an interior lattice point of the zonotope $Z(X)$.
 Then $f_z = \psi_X(\todd(X,z))\in \Pcal_-(X)$.
 \end{Lemma}

\begin{proof} 
Suppose that $X$ contains $N+1$ non-zero entries.
 $\todd(X,z)$ agrees with its Taylor expansion
 \begin{align}
  \todd(X,z) = 1 + c_1 s + \ldots + c_N s^{N} + \ldots 
 \end{align}
The coefficients $c_i$  depend on $z$ and $X$ and can be expressed in terms of Bernoulli numbers.
It is sufficient to show that $c_N=0$. 
 This can be done by calculating the residue at the origin. Let $\gamma\subseteq \CC$ be a circle around the origin. Using the residue theorem and the considerations at the beginning 
  of this section we obtain: 
\begin{align}
 c_N &=  \res\nolimits_0 (z^{-(N+2)} \todd(X,z)) = \res\nolimits_0 \left(\frac{1}{(e^s - 1)^{\alpha}(1 - e^{-s})^{\beta}   } \right)
  \\ &=   \frac{1}{2\pi i} \oint_\gamma  \frac{1}{ (1-e^{-s})^\alpha (e^s -1 )^\beta} \,\di s
  =   
  \frac{1}{2\pi i} \oint_{\exp(\gamma)}  \frac{1}{ (1 - \sigma^{-1})^\alpha ( \sigma -1)^\beta} \frac 1 \sigma \,\di \sigma
  \\ &=     
  \nonumber
  \frac{1}{2\pi i} \oint_{\exp(\gamma)}  \frac{ \sigma^{\alpha-1} }{ (\sigma - 1)^{\alpha +\beta}}  \,\di \sigma
    =   \frac{1}{2\pi i} \oint_{\exp(\gamma)}  \sum_{ i = 0}^{\alpha - 1} \binom{\alpha - 1}{i} (\sigma - 1)^{-(\alpha+ \beta) + i   } \,\di \sigma.
  \end{align}
  The last equality can easily be seen   by substituting $y = \sigma+1$, expanding, and resubstituting.
  Note that $\exp(\gamma)$ is a curve around $1$. %
  Since $\alpha$ and $\beta$ are positive the residue of the integrand of the last integral at one is $0$.
 Hence $c_N = 0 $.
\end{proof}
\end{appendix}

\renewcommand{\MR}[1]{} 

%

\providecommand{\bysame}{\leavevmode\hbox to3em{\hrulefill}\thinspace}
\providecommand{\MR}{\relax\ifhmode\unskip\space\fi MR }
\providecommand{\MRhref}[2]{%
  \href{http://www.ams.org/mathscinet-getitem?mr=#1}{#2}
}
\providecommand{\href}[2]{#2}

\end{document}